\documentclass[12pt,a4paper]{article}
\usepackage{latexsym,amsthm,amssymb,amscd,amsmath}
\usepackage[colorlinks]{hyperref}
\usepackage[left=2.2cm,bottom=2.7cm,top=2.7cm,right=2.2cm]{geometry}
\usepackage[all]{xy}
\usepackage{graphicx}
\newcounter{alphthm}
\usepackage{float}
\newtheorem{propriete}[alphthm]{Theorem}

\newtheorem{thm}{Theorem}[section]

\newtheorem{cor}{Corollary}[section]

\theoremstyle{definition}

\newtheorem{rem}{Remark}[section]
\newtheorem{ex}{Example}
\newcommand{\be}{\begin{equation}}
\newcommand{\ee}{\end{equation}}
\newcommand{\pa}{{\partial}}

\newcommand{\pxi}{{\pa \over \pa x^i}}

\title{On Berwald Spaces with non-Zero Flag Curvature}
\author{A. Tayebi and B. Najafi}
\numberwithin{equation}{section}
\begin{document}

\maketitle
\begin{abstract}
We prove that every Berwald manifold with non-zero flag curvature is Riemannian. This result provides an extension of Numata  and  Szab\'{o}'s rigidity theorems. We show that every positively curved constant isotropic Berwald manifold is Riemannian or locally Minkowskian. Then, we prove that every compact strictly positive (or negative) isotropic Berwald manifold reduces to a Berwald manifold. Finally, we prove that every   homogeneous  isotropic Berwald   metric is  either  locally Minkowskian, or Riemannian, or Berwald metric or Berwald-Randers metric generalizing result previously only known in the case of Randers metric.
\end{abstract}

\section{Introduction}
There are two types of curvatures in Finsler  geometry, i.e.,  Riemannian and non-Riemannian curvatures. These curvatures interact to each other and have some mysterious and hidden  relations between them. Among the non-Riemannian curvatures, the Berwald curvature has a central position. Finsler metrics with vanishing Berwald curvature are called Berwald metrics. Let  $(M, F)$ be  a Finsler manifold and $\gamma: [a, b]\rightarrow M$ be a piecewise $C^\infty$ curve from $\gamma(a)=p$ to $\gamma(b)=q$. For every vector $u\in T_pM$, let us define $P_\gamma:T_pM\rightarrow T_qM$ by $P_\gamma(u):=U(b)$, where $U=U(t)$ is the parallel vector field along $\gamma$ such that $U(a)=u$. $P_\gamma$ is called the parallel translation along $\gamma$. In \cite{I},  Ichijy\={o} showed that if  $F$ is a Berwald metric, then  all tangent spaces $(T_xM, F_x)$ are linearly isometric to each other.

The flag curvature ${\bf K} = {\bf K}(P, y)$ of a Finsler metric $F = F(x, y)$ on an $n$-manifold $M$ is a function of ``flag" $P\subset T_xM$
and ``flag pole" $ y\in T_xM$ at $x$ with $y\in P$. The flag curvature is a natural extension of the sectional curvature in Riemannian geometry. A Finsler metric $F$ is of scalar flag curvature, if ${\bf K}(x, y, P)={\bf K}(x, y)$ is independent of $ P$. In this case, the flag curvature is just a scalar function on the tangent space of $M$. By definition, every Berwald  metric is affinely equivalent to a Riemannian metric \cite{ShDiff}. Therefore, Berwald metrics are closeness  Finsler metrics to the Riemannian metrics.  In 1975, Numata considered Berwald metrics of scalar flag curvature and proved the following.
\begin{propriete}{\rm (Numata Theorem \cite{Num})}\label{Num}
Let $(M, F)$ be a Berwald manifold of dimension $n\geq3$.   Suppose that $F$ is of scalar
flag curvature ${\bf K}={\bf K}(x, y)$. Then, the following hold
\begin{itemize}
\item[1.] If\ \ ${\bf K}$ vanishes on $TM_0$, then $F$ is locally Minkowskian;
\item[2.] If\ \ ${\bf K}$ is no where zero on $TM_0$, then $F$ is Riemannian metric.
\end{itemize}
\end{propriete}
It is remarkable that, using Theorem \ref{Num}, Numata  provided a same result for Landsberg metrics \cite{Num}. However, Numata theorem has the restriction on the dimension of manifold. Then, in 1981, Szab\'{o} studied two-dimensional Finsler spaces and completed the Numata theorem as follows.
\begin{propriete}{\rm (Szab\'{o}  Theorem \cite{Szabo1})}\label{Szabo}
Let $(M, F)$ be a  2-dimensional Berwald  manifold. Then, the following hold
\begin{enumerate}
  \item If \ \ ${\bf K}$ vanishes on $TM_0$, then $F$ is locally Minkowskian;
  \item If \ \ ${\bf K}$ is no where zero $TM_0$, then $F$ is Riemannian;
  \end{enumerate}
More precisely, every  Berwald surface is either locally Minkowskian  or Riemannian.
\end{propriete}
By Numata and  Szab\'{o} rigidity theorems, one can find that every $n$-dimensional Berwald manifold of scalar flag curvature is Riemannian (if ${\bf K}$ is nonzero) or locally Minkowskian (if ${\bf K}$ vanishes). However, Finsler metrics of scalar flag curvature  are special Finsler metrics.  In this paper, we remove the mentioned restrictions on Finsler manifolds and prove the following rigidity result.
\begin{thm}\label{MTHM}
Let $(M, F)$ be a Berwald manifold. Then, the following hold
\begin{itemize}
\item[1.] If\ \ ${\bf K}(P, y)$ vansihes for all flags and flag-poles, then $F$ is locally Minkowskian;
\item[2.] If\ \ ${\bf K}(P, y)$ is no where zero, then $F$ is Riemannian metric.
\end{itemize}
\end{thm}
Here, we highlight two important remarks. First, Theorem \ref{MTHM}  gives a natural and important extension of theorems \ref{Num} and \ref{Szabo}. Second, as an application of Theorem \ref{MTHM}, we give a short and smart proof for Szab\'{o}  Theorem \ref{Szabo}. For more details, see Corollary \ref{CSZ}.

\bigskip

The well-known  Hilbert's fourth problem asks to construct all metrics on $n$-dimensional Euclidean space $\mathbb{R}^n$ such that the straight line segment is the shortest curve joining two points. The beautiful and important solutions for Hilbert's problem were constructed by Funk in \cite{Fk2}, namely Funk metrics, which are  the solutions of PDEs $F_{x^k}=FF_{y^k}$. Regarding the Berwald curvature of Funk metric, Cheng-Shen introduced the notion of isotropic Berwald metrics \cite{CS}. A Finsler metric $F$ is called a isotropic Berwald metric  if its Berwald curvature is given by
\begin{eqnarray}\label{IBCurve}
{\bf B}_y(u,v,w)=\Phi F^{-1}\Big\{{\bf h}(u,v)\mathfrak{B}_y(w)+{\bf h}(v, w)\mathfrak{B}_y(u) +{\bf h}(w, u)\mathfrak{B}_y(v)+2F{\bf C}_y(u, v, w) \ell\Big\},
\end{eqnarray}
where $\Phi\in C^{\infty}(M)$,  $\ell=\ell(x, y)=\ell^i\partial/\partial x^i$ denotes the distinguished section of $\pi^*TM$, $\mathfrak{B}_y(w):=w-\textbf{g}_y(w,\ell)\ell$   and ${\bf h}_y(u,v)={\bf g}_y(u,v)-F^{-2}(y){\bf g}_y(y,u){\bf g}_y(y,v)$  is  the angular form in direction $y$. If $\Phi=constant$, then $F$ is called constant isotropic Berwald metric.  Berwald metrics are trivially isotropic Berwald metrics with $\Phi=0$. Funk metrics are also non-trivial isotropic Berwald metrics  with $\Phi=1/2$. In \cite{TN}, the authors studied  isotropic Berwald metrics of scalar flag curvature and proved the following.
\begin{propriete}{\rm (\cite{TN})}\label{TNX}
Let $(M, F)$ be an isotropic Berwald manifold of dimension $n\geq 3$. Suppose that $F$ is of scalar flag curvature. Then $F$ is a Randers metric.
\end{propriete}
In this paper, we exclude the class of Berwlad manifolds and consider the class of positively curved non-zero constant isotropic Berwald manifolds. Then, prove the following.

\begin{thm}\label{MTHM2}
Let $(M, F)$ be a positively curved non-zero constant isotropic Berwald manifold. Then, $F$ is Riemannian metric.
\end{thm}
Theorem \ref{MTHM2} can be considered as an extension of Theorem \ref{TNX}. Also, we remark that Theorem \ref{MTHM2} does not hold for constant isotropic Berwald manifold of negative flag curvature. Here, we  give an example of $n$-dimensional  constant isotropic  Berwald manifold with negative constant flag curvature  that is not Riemannian.
\begin{ex}
As we mentioned, a Finsler metric $F=F(x, y)$ satisfying $F_{x^k}=FF_{y^k}$ is called a Funk metric. The standard Funk metric on the
Euclidean unit ball $\mathbb{B}^n(1)$ is denoted by $\Theta$ and defined by
\begin{equation}\label{Funk}
\Theta(x,y):=\frac{\sqrt{|y|^2-(|x|^2|y|^2-\langle x, y \rangle^2)}}{1-|x|^2}+\frac{\langle x, y \rangle}{1-|x|^2}, \,\,\,\,\,y\in T_x\mathbb{B}^n(1)\simeq \mathbb{R}^n,
\end{equation}
where $\langle, \rangle$ and $| . |$ denote the Euclidean inner product and norm on $\mathbb{R}^n$, respectively. The spray coefficients of  $\Theta$  are given by
\[
G^i=\frac{1}{2} \Theta y^i.
\]
Then, Funk metrics are of constant isotropic  Berwald curvature $\Phi=1/2$. Theses metrics have  negative constant flag curvature ${\bf K}=-1/4$. Funk metrics are of Randers-type Finsler metrics which are not Riemannian.
\end{ex}

\bigskip

Let $(M, F)$ be an $n$-dimensional Finsler manifold. Then, $F$ is called of strictly positive (or negative) isotropic Berwald curvature if its satisfies \eqref{IBCurve}, where  $\Phi(x)\ge c_0>0$ (or $\Phi(x)\leq c_0<0$) on $M$. In this paper, we consider strictly positive (or negative) isotropic Berwald manifolds and prove the following.

\begin{thm}\label{MTHM2.5}
Let $(M,F)$ be a compact  Finsler manifold. Suppose that $F$ has strictly positive (or negative) isotropic Berwald curvature. Then $F$ is a Berwald metric.
\end{thm}

\bigskip

For a  Finsler manifold $(M, F)$, Deng-Hou  proved that the group of isometries $I(M, F)$  is a Lie transformation group of  manifold $M$ which is useful for studying homogeneous Finsler manifold  \cite{D2}. Homogeneous Finsler manifolds are those Finsler manifolds $(M, F)$ that the orbit of the natural  action  of $I(M, F)$ on $M$ at any point of $M$ is the whole $M$.  In 2013, Deng-Hu investigated homogeneous Randers manifolds of Berwald-type and proved the following.
\begin{propriete}{\rm (Deng-Hu  Theorem \cite{DH})}\label{DH}
Let $(M, F)$ be a homogeneous Randers space of Berwald type. If the flag curvature of $F$ is no where zero, then $F$ is Riemannian.
\end{propriete}
By considering Theorems \ref{DH} and \ref{MTHM2}, it is a nature problem that whether there are  non-Berwald homogeneous Finsler metrics with isotropic Berwald curvature or not. Then, we give a generalization of Deng-Hu  Theorem \ref{DH}. More precisely, we prove the following.
\begin{thm}\label{TN}
Every   homogeneous  isotropic Berwald   metric on a connected manifold $M$ is  either  locally Minkowskian, or Riemannian, or Berwald metric or Randers metric of Berwald-type.
\end{thm}

\section{Preliminary}
Let $M$ be an $n$-dimensional $C^{\infty}$ manifold, $TM=\bigcup_{x \in M}T_{x}M$ the tangent space and $TM_0:=TM-\{0\}$ the slit tangent space of $M$. A Finsler structure on manifold $M$ is a function $ F:TM\rightarrow [0,\infty )$ with the following properties: (i) $F$ is $C^\infty$ on $TM_0$; (ii) $F$ is positively 1-homogeneous on the fibers of tangent  bundle $TM$, i.e., $F(x,\lambda y)=\lambda F(x,y)$, $\forall \lambda>0$;  (iii)   The   quadratic form $\textbf{g}_y:T_xM \times T_xM\rightarrow \mathbb{R}$ is positive-definite on $T_xM$
\[
\textbf{g}_{y}(u,v):={1 \over 2}\frac{\partial^2}{\partial s\partial t} \Big[  F^2 (y+su+tv)\Big]_{s=t=0}, \ \
u,v\in T_xM.
\]
Then, the pair  $(M,F)$ is called a  Finsler manifold.

Let  $x\in M$ and $F_x:=F|_{T_xM}$.  To measure the non-Euclidean feature of $F_x$, one can define ${\bf C}_y:T_xM\times T_xM\times T_xM\rightarrow \mathbb{R}$ by
\[
{\bf C}_{y}(u,v,w):={1 \over 2} \frac{d}{dt}\Big[\textbf{g}_{y+tw}(u,v)
\Big]_{t=0}, \ \ u,v,w\in T_xM.
\]
The family ${\bf C}:=\{{\bf C}_y\}_{y\in TM_0}$  is called the Cartan torsion. It is well known that ${\bf{C}}=0$ if and
only if $F$ is Riemannian.

There is a weaker notion of Cartan torsion. For   $y\in T_x M_0$, define ${\bf I}_y:T_xM\rightarrow \mathbb{R}$ by
\[
{\bf I}_y(u):=\sum^n_{i=1}g^{ij}(y) {\bf C}_y(u, \partial_i, \partial_j),
\]
where $\{\partial_i\}$ is a basis for $T_xM$ at  $x\in M$. The family ${\bf I}:=\{{\bf I}_y\}_{y\in TM_0}$  is called the mean Cartan torsion.  By Deicke's  theorem, every positive-definite Finsler metric  $F$ is Riemannian  if and only if ${\bf I}_y=0$

\bigskip
Let $(M, F)$ be a  Finsler manifold. A global vector field ${\bf{G}}=y^i {{\partial} / {\partial x^i}}-2G^i(x,y){{\partial} /
{\partial y^i}}$ is
induced by a Finsler metric $F$, where $G^i$ are local functions on $TM$ given by
\[
G^i:=\frac{1}{4}g^{im}\Big\{\frac{\partial^2[F^2]}{\partial x^k
\partial y^m}y^k-\frac{\partial[F^2]}{\partial x^m}\Big\},\ \
y\in T_xM.
\]
We call $\bf{G}$  the  associated spray of $(M,F)$.

A natural volume form $dV_F = \sigma_F(x) dx^1 \cdots
dx^n$ of a Finsler metric $F$ on an $n$-dimensional manifold $M$ is defined by
\[
\sigma_F(x) := {{\rm Vol} (\Bbb B^n)
\over {\rm Vol} \Big \{ (y^i)\in \mathbb{R}^n \ \Big | \ F \big ( y^i
\pxi|_x \big ) < 1 \Big \} },
\]
where $\Bbb B^n=\{y\in\mathbb{R}^n|\,\, |y|<1\}$.  The S-curvature is defined by
\[
 {\bf S}({\bf y}) := {\pa G^i\over \pa y^i}(x,y) - y^i {\pa \over \pa x^i}
\Big [ \ln \sigma_F (x)\Big ],
\]
where ${\bf y}=y^i\partial/\partial x^i|_x\in T_xM$.  $F$ is called of  isotropic S-curvature if ${\bf S}= (n+1) c F$, where $c= c(x)$ is a scalar function and  on $M$.

\smallskip

Define ${\bf B}_y:T_xM\times T_xM \times T_xM\rightarrow T_xM$  by ${\bf B}_y(u, v, w):=B^i_{\ jkl}(y)u^jv^kw^l{{\partial }/ {\partial x^i}}|_x$, where
\[
B^i_{\ jkl}:={{\partial^3 G^i} \over {\partial y^j \partial y^k \partial y^l}}.
\]
$\bf B$ is called the Berwald curvature, and $F$ is called a Berwald metric if ${\bf{B}}=0$.

\bigskip
For $y\in T_xM$, define the Landsberg curvature ${\bf L}_y:T_xM\times T_xM \times T_xM\rightarrow \mathbb{R}$  by
\[
{\bf L}_y(u, v,w):=-\frac{1}{2}{\bf g}_y\big({\bf B}_y(u, v, w), y\big).
\]
${\bf L}_y(u, v, w)$ is symmetric in $u$, $v$, and $w$, and  ${\bf L}_y(y, v, w)=0$. $\bf L$ is called the
Landsberg curvature.
\bigskip

For $y\in T_xM$, define ${\bf J}_y:T_xM\rightarrow \mathbb{R}$  by ${\bf J}_y(u):=J_i(y) u^i$, where
\[
{\bf J}_y(u):=\sum^n_{i=1}g^{ij}(y) {\bf L}_y(u, \partial_i, \partial_j).
\]
By definition,   ${\bf J}_y(y)=0$. $\bf J$ is called the mean Landsberg curvature or J-curvature.
A Finsler metric $F$ is called a weakly Landsberg  metric if ${\bf J}_y=0$.

\bigskip

For a non-zero vector $y \in T_xM_{0}$, the Riemann curvature is a family of linear
transformation $\textbf{R}_y: T_xM \rightarrow T_xM$ with homogeneity ${\bf R}_{\lambda y}=\lambda^2 {\bf R}_y$,
$\forall \lambda>0$ which is defined by
$\textbf{R}_y(u):=R^i_{k}(y)u^k {\partial/ {\partial x^i}}$, where
\be
R^i_{k}(y)=2{\partial G^i \over {\partial x^k}}-{\partial^2 G^i \over
{{\partial x^j}{\partial y^k}}}y^j+2G^j{\partial^2 G^i \over
{{\partial y^j}{\partial y^k}}}-{\partial G^i \over {\partial
y^j}}{\partial G^j \over {\partial y^k}}.\label{Riemannx}
\ee
The family $\textbf{R}:=\{\textbf{R}_y\}_{y\in TM_0}$ is called the Riemann curvature.
\bigskip

For a flag $P:={\rm span}\{y, u\} \subset T_xM$ with the flagpole $y$, the flag curvature ${\bf
K}={\bf K}(P, y)$ is defined by
\be
{\bf K}(x, y, P):= {{\bf g}_y \big(u, {\bf R}_y(u)\big) \over {\bf g}_y(y, y) {\bf g}_y(u,u)
-{\bf g}_y(y, u)^2 }.\label{K}
\ee
The flag curvature ${\bf K}={\bf K}(x, y, P)$ is a function of tangent planes $P={\rm span}\{ y, v\}\subset T_xM$. This quantity tells us how curved the space is at a point.   A Finsler metric $F$ is of scalar flag curvature, if ${\bf K}(x, y, P)={\bf K}(x, y)$ is independent of $ P$. In this case, the flag curvature is just a scalar function on the tangent space of $M$.

\section{Proof of Theorem \ref{MTHM}}
Express the Riemann curvature by
${\bf R}= R^i_{\ k} \omega^k \otimes {\bf e}_i$.
Let
\be
R^i_{\ kl}:= {1\over 3} \Big \{\frac{\partial R^i_{\ k}}{\partial y^l} - \frac{\partial R^i_{\ l}}{\partial y^k} \Big \}.\label{Kikl}
\ee
In fact,  $R^i_{\ kl}$ and $L^i_{\ kl}:= g^{ij}L_{jkl}$ are determined by the following equation
\[
\Omega^i := d\omega^{n+i}-\omega^{n+j} \wedge \omega_j^{\ i}
= {1\over 2} R^i_{\  kl}\omega^k \wedge \omega^l - L^i_{\ kl} \omega^k \wedge \omega^{n+l}.
\]
In this paper, we use the Berwald connection  and the $h$- and $v$- covariant derivatives of a Finsler tensor field are denoted by `` $|$ " and ``, " respectively. In local coordinate system, the Berwald connection determined by following
\begin{eqnarray}
&&d\omega^i = \omega^j \wedge \omega^i_j,\\
&&dg_{ij}- g_{kj}\omega^k_i - g_{ik}\omega^k_j = -2L_{ijk}\omega^k + 2C_{ijk}\omega^{n+k}.\label{Berwaldms}
\end{eqnarray}
Thus
\[
g_{ij|k}= -2L_{ijk},   \ \ \ \ \ g_{ij,k}= 2C_{ijk}.
\]
Also, we have
\begin{equation}
\Omega^i_{\ j}=d\omega^i_{\ j}-\omega^k_{\ j}\wedge\omega^i_{\
k}=\frac{1}{2}R^i_{\ jkl}\omega^k\wedge\omega^l-B^i_{\
jkl}\omega^k\wedge\omega^{n+l}.\label{BCon}
  \end{equation}
Thus
\begin{eqnarray}
&&R^h_{\ ijk}+ R^h_{\ jki}+ R^h_{\ kij}=0,\label{xxxzzz}\\
&&B^h_{\ ijk}=B^h_{\ jki}=B^h_{\ kij}.
\end{eqnarray}
Differentiating (\ref{Berwaldms}) implies that
\begin{eqnarray}
\nonumber g_{pj}\Omega^p_i+g_{ip}\Omega^p_j=\!\!\!\!&-&\!\!\!\!\! 2L_{ijk|s}\omega^k\wedge\omega^s-2L_{ijk,s}\omega^k\wedge\omega^{n+s}\\
\!\!\!\!&-&\!\!\!\!\! 2C_{ijs|k}\omega^k\wedge\omega^{n+s}-2C_{ijs,k}\omega^{n+k}\wedge\omega^{n+s}-2C_{ijs}\Omega^s.\label{mohem2}
\end{eqnarray}
Putting (\ref{BCon}) in (\ref{mohem2}) gives us
\begin{eqnarray}
&&L_{ijk|l}-L_{ijl|k}=-\frac{1}{2}g_{pj}R^p_{i\ kl}-\frac{1}{2}g_{ip}R^p_{j\ kl}-C_{ijl}R^l_{\ kl}.\label{xxxyyy}
\end{eqnarray}
By contracting (\ref{xxxyyy}) with $y^l$, we get
\be
L_{ijk;m}y^m + C_{ijm}R^m_{\ \ k}
= - {1\over 2} g_{im} R^m_{\ \ kl\cdot j} y^l - {1\over 2} g_{jm} R^m_{\ \ kl\cdot j} y^l,\label{LCKK}
\ee
Plugging (\ref{Kikl}) into (\ref{LCKK}) yields
\begin{eqnarray}
L_{ijk|m}y^m + C_{ijm}R^m_{\ \ k}=   - {1\over 3}g_{im}R^m_{\ \ k\cdot j}
- {1\over 3} g_{jm} R^m_{\ \ k\cdot i}- {1\over 6} g_{im}R^m_{\ \ j\cdot k}
- {1\over 6} g_{jm}R^m_{\ \ i\cdot k}. \label{Moeq1}
\end{eqnarray}
For more details about the relation \eqref{Moeq1}, one can refer to \cite{MS}.

\bigskip

\noindent
{\bf Proof of Theorem \ref{MTHM}:} Since $ {\bf K}$ is no where zero and it is a continuous function, either ${\bf K} <0$ for all flags and flag-poles or ${\bf K} >0$  for all flags and flag-poles.  Assume that at a point $x\in M$ and for all flags and flag-poles, we have ${\bf K} <0$.

Now, by multiplying (\ref{Moeq1}) with $g^{ij}$  we get
\be
J_{k|m}y^m + I_mR^m_{\ k}  = -  {1\over 3} \Big \{ 2 R^m_{\ \; k\cdot m} + R^m_{\ \; m\cdot k} \Big \} \label{Moeq2}
\ee
Also, S-curvature satisfies the following equation
\be
{\bf S}_{\cdot k|m}y^m - {\bf S}_{| k}
= - {1\over 3} \Big \{ 2 R^m_{\ \; k\cdot m} + R^m_{\ \; m\cdot k} \Big \}.\label{SSKKMo}
\ee
For more details, see \cite{CMS}. Comparing (\ref{Moeq2}) and (\ref{SSKKMo}) yields
\be
J_{k|m}y^m +I_m R^m_{\ k} = {\bf S}_{\cdot k|m}y^m - {\bf S}_{|k}. \label{EIILJ}
\ee
See \cite{Shn}. We can rewrite \eqref{EIILJ} as follows
\be
I^i_{\ | p|q} y^py^q +  R^i_{\ m}I^m = g^{ik} \Big \{ {\bf S}_{\cdot k|m}y^m - {\bf S}_{|k} \Big \}. \label{EIILJ*}
\ee
By assumption, ${\bf B}=0$. Every Berwald metric satisfies ${\bf S}=0$ and ${\bf J}=0$. Then, \eqref{EIILJ*} reduces to
\be
{\bf R}_y({\bf I}_y)=0. \label{RI}
\ee
Since ${\bf I}_y$ is orthogonal to $y$ with respect to ${\bf g}_y$, then by \eqref{RI}, it follows that  ${\bf K}(P_0, y)=0$, where
$P_0= {\rm span} \{{\bf I}_y, y\}$ whenever ${\bf I}_y \not=0$. But by assumption, we have ${\bf K}(P_0, y)<0$. This contradiction yields that   ${\bf I}_y =0$ for all $y\in T_xM_0$.

The same argument works for the case ${\bf K} > 0$. Now, Deicke's theorem infers $F$ is Riemannian.
\qed

\bigskip

\begin{cor}{\rm (Szab\'{o} Rigidity Theorem)}\label{CSZ}
Let $(M, F)$ be a Berwald surface. Then $F$ is Riemannian metric or locally Minkowskian.
\end{cor}
\begin{proof}
Let $F$ be a Berwald surface (or of scalar flag curvature), $R^i_{\ m}={\bf K}F^2h^i_m$. Then \eqref{RI} reduces to ${\bf K}I_i =0$. Since ${\bf B}=0$, then by Akbar-Zadeh theorem ${\bf K}={\bf K}(x)$ (see \cite{AZ}). Let  ${\bf K}(x)\neq 0$, $\forall x\in M$. Then ${\bf I}=0$ and $F$ is a Riemannian metric.  If $F$ is not Riemannian, then ${\bf K}(x)= 0$. Every Berwald metric with vanishing flag curvature is locally Minkowskian.
\end{proof}

\bigskip
Let $(M, F)$ be an $n$-dimensional  Finsler manifold.  Then  $F$ is called of almost isotropic S-curvature if ${\bf S}= (n+1) c F+\eta$, where $c= c(x)$ is a scalar function and  $\eta=\eta_i(x)y^i$ is a $1$-form on $M$. If $c=constant$, then $F$ is called of almost constant isotropic S-curvature.  Also, $F$ is called of isotropic mean Berwald curvature, if ${\bf E}=\frac{n+1}{2}cF^{-1}{\bf h}$, where $c= c(x)$ is a scalar function  on $M$. If $c=constant$, then $F$ is called of constant $E$-curvature. By definition, ${\bf E}=\frac{n+1}{2}cF^{-1}{\bf h}$ if and only if  ${\bf S}=(n+1)cF+\eta$. Then, $F$ has almost constant isotropic S-curvature if and only if it has constant $E$-curvature. In \cite{Wu},  Wu proved that any closed weakly Landsberg manifold with negative flag curvature is Riemannian.  It is interesting if one can replace the closeness condition with other curvature property.  By  relation \eqref{EIILJ*} used in proof of  Theorem \ref{MTHM}, we conclude the following.
\begin{cor}\label{cor-1}
Let $(M, F)$ be a weakly Landsberg  manifold. Suppose that $F$ has constant mean Berwald curvature. Then, the following hold
\begin{itemize}
\item[1.] If\ \ ${\bf K}(P, y)$ vansihes for all flags and flag-poles, then $F$ is locally Minkowskian;
\item[2.] If\ \ ${\bf K}(P, y)$ is no where zero, then $F$ is Riemannian metric.
\end{itemize}
\end{cor}
\section{Proof of Theorem \ref{MTHM2}}
In this section, we study positively curved non-zero constant isotropic Berwald manifolds and prove Theorem \ref{MTHM2}.

\bigskip
\noindent
{\bf Proof of Theorem \ref{MTHM2}:} Let $F$ be a constant isotropic Berwald   metric on an $n$-dimensional manifold $M$. Then, the Berwald curvature of $F$ is given by
\begin{eqnarray}\label{IBCc}
\nonumber{\bf B}_y(u,v,w)&=& c F^{-1}\Big\{{\bf h}_y(u,v)\big (w-\textbf{g}_y(w,\ell)\ell\big)+{\bf h}_y(v, w)\big(u-\textbf{g}_y(u,\ell)\ell\big)
\\ &&\quad\quad \ \ \ \ \ \ \ \ \  \ \ \ \ \ \ \ +{\bf h}_y(w, u)\big(v-\textbf{g}_y(v,\ell)\ell\big)+2F{\bf C}_y(u, v, w) \ell\Big\}.
\end{eqnarray}
where $c$ is a real constant.  By assumption, we have $c\neq 0$. Applying  ${\bf g}_y(y,\cdot)$ on both sides of \eqref{IBCc}   implies that
\be
{\bf L}=c F {\bf C}.\label{L}
\ee
Taking a trace of  \eqref{L} yields
\be
{\bf J}=c F {\bf I}.\label{J}
\ee
Taking horizontal derivation of \eqref{J} along Finslerian geodesics gives us
\be
{\bf J}'=cF{\bf J}=c^2F^2 {\bf I}.\label{JI}
\ee
Also, every isotropic Berwald   metric \eqref{IBCc} has isotropic $S$-curvature, ${\bf S}=(n+1)c F$.  By assumption, we have $c=constant$. Then, we have
\be
{\bf S}=(n+1)c F,   \ \ \ \ c\in \mathbb{R}.\label{SS}
\ee
Also, the following hold
\[
F_{|m}=0, \ \ \ \ FF_{\cdot k |m}=g_{ik|m}y^i=-2L_{ikm}y^i=0.
\]
Then, by \eqref{SS},  we get
\be
{\bf S}_{\cdot k|m}y^m - {\bf S}_{|k}=0, \label{SC}
\ee
By  \eqref{EIILJ*}, \eqref{JI} and \eqref{SC} we get
\be
c^2F^2{\bf I}_y+  {\bf R}_y({\bf I}_y) = 0. \label{IR}
\ee
Consider $P_0= {\rm span} \{{\bf I}_y, y\}$ whenever ${\bf I}_y \not=0$. Applying ${\bf g}_y \big({\bf I}_y, \cdot\big)$ on both sides of  \eqref{IR}, we get
\be
{\bf g}_y \big({\bf I}_y, {\bf I}_y\big)F^2\Big (c^2+  {\bf K}(P_0, y) \Big )= 0, \label{IRR}
\ee
where we have used
\[
{\bf g}_y \big({\bf I}_y, {\bf R}_y({\bf I}_y)\big)=F^2{\bf g}_y \big({\bf I}_y, {\bf I}_y\big){\bf K}(P_0, y).
\]
By \eqref{IRR}, we have
\be
 {\bf K}(P_0, y) = -c^2, \label{IRRR}
\ee
which is a contradiction with our assumption on ${\bf {K}}$. Thus, ${\bf I}=0$ and by  Deicke's theorem $F$ reduces to a  Riemannian metric.
\qed

\bigskip

It is proved that every Douglas metric with isotropic (constant) mean Berwald curvature  is a isotropic (constant) Berwald metric (see Lemma 4.1 in \cite{CS}). Then, by Theorem \ref{MTHM2}, we get the following.
\begin{cor}
Let $(M, F)$ be a positively curved Douglas manifold. Suppose that $F$ has  non-zero constant isotropic mean Berwald curvature. Then, $F$ is Riemannian metric.
\end{cor}

\section{Proof of Theorem \ref{MTHM2.5}}
\bigskip
\noindent
{\bf Proof of Theorem \ref{MTHM2.5}:} Let $F=F(x,y)$ be a Finsler metric on an $n$-dimensional manifold $M$. The distortion $\tau =\tau(x,y) $ on $TM$
associated with the  Busemann-Hausdorff volume form $dV_{BH}=\sigma(x) dx$ is defined by
\[
\tau  (x,y) = \ln \frac{ \sqrt{ \det \big(g_{ij}(x,y)\big) }}{ \sigma_F(x) }.
\]
Let $SM=\Big\{(x,y)\in TM:\ F(x,y)=1\Big\}$ be the unit sphere bundle. Since $M$ is a compact manifold,
$SM$ is compact, and hence $\tau$ is a bounded function on $SM$. In particular, there exists a real constant
$C>0$ such that
\[
|\tau(x,y)|\le C \qquad  \ \forall(x,y)\in SM.
\]
Now, let $\gamma=\gamma(t)$ be any unit-speed geodesic, so $F(\dot\gamma(t))\equiv 1$ and
$(\gamma(t),\dot\gamma(t))\in SM$ for all $t$. By the definition of $S$--curvature,
\[
\frac{d}{dt}\tau\big(\gamma(t),\dot\gamma(t)\big)={\bf S}\big(\gamma(t),\dot\gamma(t)\big).
\]
By assumption, $F$  has isotropic Berwald curvature \eqref{IBCurve}. Then, it has isotropic $S$-curvature
\[
{\bf S}=(n+1)\Phi F.
\]
Using the isotropy and $F\big(\dot\gamma(t)\big)\equiv 1$, we get
\[
\frac{d}{dt}\tau\big(\gamma(t),\dot\gamma(t)\big)= (n+1)\Phi\big(\gamma(t)\big).
\]
Integrating from $0$ to $t$ yields
\[
\tau\big(\gamma(t),\dot\gamma(t)\big)-\tau\big(\gamma(0),\dot\gamma(0)\big)= (n+1)\int_0^t \Phi\big(\gamma(s)\big)\,ds.
\]
By the assumption $\Phi(\cdot)\ge c_0>0$ on $M$,
\[
\int_0^t \Phi\big(\gamma(s)\big)\,ds \;\ge\; c_0\,t,
\]
so
\[
\tau\big(\gamma(t),\dot\gamma(t)\big)-\tau\big(\gamma(0),\dot\gamma(0)\big)\;\ge\; (n+1)c_0\,t.
\]
Therefore
\[
\tau\big(\gamma(t),\dot\gamma(t)\big) \;\ge\; \tau\big(\gamma(0),\dot\gamma(0)\big) + (n+1)c_0\,t,
\]
which tends to $+\infty$ as $t\to\infty$. This contradicts the boundedness
$|\tau|\le C$ on $SM$. Hence no such $c_0>0$ can exist, and the only possibility
under $\Phi\ge 0$ is $\Phi\equiv 0$.
\qed

\bigskip
By  Theorem \ref{MTHM2.5}, we conclude the following.
\begin{cor}
Every compact  constant isotropic Berwald manifold is a Berwald manifold.
\end{cor}

\section{Proof of Theorem \ref{TN}}
A Randers metric  on a manifold $M$ is a positive scalar function on $TM$ defined by $F=\alpha +\beta$,  where   $\alpha=\sqrt{a_{ij}(x)y^iy^j}$ is a Riemannian metric and $\beta =b_i(x)y^i$ is a 1-form on  $M$.  This metric was introduced by Randers in the context of general relativity \cite{ShDiff}. In \cite{DH}, Deng and Hu proved that a homogeneous Randers metric of Berwald type whose flag curvature is  non-zero everywhere must be Riemannian. In this section, we prove Theorem \ref{TN}. First, we remark some notions.

\bigskip
\noindent
{\bf Proof of Theorem \ref{TN}:} Let $F$ be homogeneous  isotropic Berwald   metric on an $n$-dimensional manifold $M$:
\begin{eqnarray}\label{IBCurve}
\nonumber{\bf B}_y(u,v,w)&=& \Phi F^{-1}\Big\{{\bf h}(u,v)\big (w-\textbf{g}_y(w,\ell)\ell\big)+{\bf h}(v, w)\big(u-\textbf{g}_y(u,\ell)\ell\big)
\\ &&\quad\quad \ \ \ \ \ \ \ \ \  \ \ \ \ \ \ \ +{\bf h}(w, u)\big(v-\textbf{g}_y(v,\ell)\ell\big)+2F{\bf C}_y(u, v, w) \ell\Big\}.
\end{eqnarray}
where $\Phi\in C^{\infty}(M)$.  We prove the result in two main cases:\\\\
{\bf Case (i): $dim(M)=2$.} In \cite{TR}, it is showed that every isotropic Berwald   metric \eqref{IBCurve} has isotropic $S$-curvature, ${\bf S}=3\Phi F$. In \cite{XD}, it is proved that homogeneous Finsler metrics with isotropic $S$-curvature has vanishing $S$-curvature, $\Phi=0$. Putting this in \eqref{IBCurve} implies ${\bf B}=0$ and $F$ is a Berwald metric. In \cite{Szabo1}, Szab\'{o} proved that any connected Berwald surface  is  locally Minkowskian or Riemannian.\\\\
{\bf Case (ii): $dim(M)\geq 3$.} In \cite{TN2},  it is proved that every homogeneous  isotropic Berwald   metric on a manifold $M$ of dimension $n\geq 3$ is a  Berwald metric or Randers metric of Berwald-type.
\qed

\bigskip
\begin{rem}
By considering the case (i), Theorem \ref{TN} completes  the Theorem 1.1 in \cite{TN2} which proved only for homogeneous  isotropic Berwald manifolds of dimension $n\geq 3$. In \cite{TN2}, we guessed that theorem \ref{TN} holds for 2-dimensional Finsler manifolds but we  did not find a proof for our conjecture.
\end{rem}


\bigskip

\noindent
Akbar Tayebi\\
Department of Mathematics, Faculty  of Science\\
University of Qom \\
Qom, Iran\\
Email:\ akbar.tayebi@gmail.com

\bigskip

\noindent
Behzad Najafi\\
Department of Mathematics and Computer Sciences\\
Amirkabir University (Tehran Polytechnic)\\
Tehran. Iran\\
Email:\ behzad.najafi@aut.ac.ir

\end{document}